\documentclass[VANCOUVER,STIX1COL]{WileyNJD-v2}

\articletype{Article Type}%

\usepackage{amsmath} %declarar nuevas ordenes
\usepackage{mathtools}
\usepackage{xy}
\xyoption{all}

\usepackage{graphicx} % Para imagenes
\usepackage{xcolor} % Para cambiar el color del texto

\usepackage{listings}
\usepackage{algpseudocode}
\usepackage{algorithm}

\newcommand{\GL}{\mathrm{GL}}

\newcommand{\spn}{\mathrm{span}}

\newcommand{\id}{\mathrm{id}}
\newcommand{\tr}{\mathrm{tr}}

\newcommand{\R}{{\mathbb{R}}}

\newcommand{\N}{{\mathbb{N}}}

\newcommand{\ff}{\mathbf{f}}
\newcommand{\xx}{\mathbf{x}}
\newcommand{\bb}{\mathbf{b}}

\begin{document}

%%%%%%%%%%%%%%%%%%%%%%%%%%%%%%%%%%%%%%%%%%%%
%             DATOS PORTADA
%%%%%%%%%%%%%%%%%%%%%%%%%%%%%%%%%%%%%%%%%%%%

\title{Structure and approximation properties of Laplacian-like matrices}

\author[1]{J. Alberto Conejero} %*

\author[2]{Antonio Falc\'o}

\author[3]{Mar\'ia Mora-Jim\'enez}

%%%%%%%%%%%%%%%%%%%%%%%%%%%%%%%%%%%%%%%%
\authormark{J.A. Conejero, A. Falc\'o, and M. Mora-Jim\'enez} 
%%%%%%%%%%%%%%%%%%%%%%%%%%%%%%%%%%%%%%%%

\address[1,3]{\orgdiv{Instituto Universitario de Matemática Pura y Aplicada}, \orgname{Universitat Polit\`ecnica de Val\`encia}, \orgaddress{\state{Cam\'{\i} de Vera, s/n, 46022 Val\`encia}, \country{Spain}}}

\address[2]{\orgdiv{ESI International Chair@CEU‐UCH, Departamento de Matemáticas, Física y Ciencias Tecnológicas}, \orgname{Universidad Cardenal Herrera‐CEU, CEU Universities}, \orgaddress{\state{San Bartolomé 55, 46115 Alfara del Patriarca, Valencia}, \country{Spain}}}

%\address[3]{\orgdiv{Org Division}, \orgname{Org Name}, \orgaddress{\state{State name}, \country{Country name}}}

%\corres{*Corresponding author name, This is sample corresponding address. \email{authorone@gmail.com}}

%\presentaddress{This is sample for present address text this is sample for present address text}

\abstract[Abstract]{Many of today's problems require techniques that involve the solution of arbitrarily large systems $A\xx=\bb$. A popular numerical approach is the so-called Greedy Rank-One Update Algorithm, based on a particular tensor decomposition. The numerical experiments support the fact that this algorithm converges especially fast when the matrix of the linear system is Laplacian-Like. These matrices that follow the tensor structure of the Laplacian operator are formed by sums of Kronecker product of matrices 
following a particular pattern. Moreover, this set of matrices is not only a linear subspace it is a a Lie sub-algebra of a matrix Lie Algebra. In this paper, we characterize and give the main properties of this particular class of matrices. Moreover, the above results allow us to propose 
an algorithm to explicitly compute the orthogonal projection onto this subspace
of a given square matrix $A \in \R^{N\times N}.$}


%\jnlcitation{\cname{%
%\author{J.A. Conejero}, 
%\author{A. Falcó}, and
%\author{M. Mora}} (\cyear{2022}), 
%\ctitle{Structure and approximation properties of Laplacian-like matrices}, \cjournal{Numer Linear Algebra Appl.}, \cvol{2022;29:1--6}.}
%%%%%%%%%%%%%%%%%%%%%%%%%%%%%%%%%%%%%%%%%%%
%%%%%%%%%%%%%%%%%%%%%%%%%%%%%%%%%%%%%%%%%%%

\maketitle

\footnotetext{\textbf{Abbreviations:} ALS, Alternating Least Square; GROU, Greedy Rank-One Updated; PGD, Proper Generalized Decomposition}

\section{Introduction}\label{sec1}

The study of linear systems is a problem that dates back to the time of the Babylonians, who used words like `length' or `width' to designate the unknowns without being related to measurement problems. The Greeks also solved some systems of equations, but using geometric methods \cite{history1}. Over the years, mechanisms to solve linear systems continued to be developed until the discovery of iterative methods, the practice of which began at the end of the 19th century, by the hand of the mathematician Gauss. The development of computers in the mid-20th century prompted numerous mathematicians to delve into the study of this problem \cite{history2, history3}.\medskip

Nowadays, linear systems are widely used to approach computational models in applied sciences, for example, in mechanics, after the discretization of a partial differential equation. There are, in the literature, numerous mechanisms to deal with this type of problem, such as matrix decompositions (QR decomposition, LU decomposition), iterative methods (Newton, quasi-Newton, \ldots), and optimization algorithms (stochastic gradient descendent, alternative least squares,\ldots), among others, see for instance \cite{leiserson2001introduction,strang2006linear,golub2013matrix}. However, most of them lose efficiency as the size of the matrices or vectors involved increases. This effect is known as the \textit{curse of the dimensionality problem}.\medskip 

To try to solve this drawback, we can use tensor-based algorithms \cite{Anthony}, since their use significantly reduces the number of operations that we must employ.  For example, we can obtain a matrix of size $100 \times 100$ (i.e. a total of $10.000$ entries), from two matrices of size $10 \times 10$ multiplied, by means the tensor product, $100 + 100 = 200$ entries \cite{graham}.\medskip

Among the algorithms based on tensor products strategies \cite{simoncini}, the Proper Generalized Decomposition (PGD) family, based on the so-called Greedy Rank-One Updated (GROU) algorithm \cite{greedy, introgreedy}, is one of the most popular techniques.  PGD methods can be interpreted as `a priori' model reduction techniques because they provide a way for the `a priori' construction of optimally reduced bases for the representation of the solution. In particular, they impose a separation of variables to approximate the exact solution of a problem without knowing, in principle, the functions involved in this decomposition \cite{PGD, PGD2}. The GROU procedure in the pseudocode is given in the Algorithm \ref{GROU_alg} (where $\otimes$ denotes the Kronecker product, that is briefly introduced in Section \ref{sectionGROA}).\medskip

\begin{algorithm}[h]
	\caption{Greedy Rank-One Update}\label{GROU_alg}
	\begin{algorithmic}[1]
		\Procedure{GROU}{$\mathbf{b}\in \mathbb{R}^{n_1\cdots n_d}, A \in \mathbb{R}^{n_1\cdots n_d\times n_1\cdots n_d}, \varepsilon > 0,\texttt{tol},\texttt{rank\_max}$}
		%\Comment{The g.c.d. of a and b}
		\State $\mathbf{r}_0 = \mathbf{b}$
		\State $\mathbf{x} = \mathbf{0}$
		\For{$i=0,1,2,\ldots,\texttt{rank\_max}$}
		\State $\textbf{y} = \arg \min_{ \mathbf{y} = \mathbf{y}_1 \otimes \cdots \otimes \mathbf{y}_d} \|\mathbf{r}_i -A \mathbf{y}\|_2^2$
		\State $\textbf{r}_{i+1} = \textbf{r}_i -A \mathbf{y}$ 
		\State $\mathbf{x}\leftarrow \mathbf{x} + \mathbf{y}$ 
		\If{$\|\mathbf{r}_{i+1} \|_2 < \varepsilon$ or $|\|\mathbf{r}_{i+1} \|_2
			-\|\mathbf{r}_{i} \|_2| < \texttt{tol}$} \textbf{goto} 13
		\EndIf
		\EndFor
		\State \textbf{return} $\mathbf{u}$ and $\|\mathbf{r}_{\texttt{rank\_max}}\|_2.$
		\State \textbf{break}
		\State \textbf{return} $\mathbf{u}$ and $\|\mathbf{r}_{i+1} \|_2$
		\EndProcedure
	\end{algorithmic}
\end{algorithm}

A good example is provided by the Poisson equation $-\Delta \phi = \ff$. Let us consider the following problem in $3$D,
\begin{equation}\label{poisson2}
\left\{ \,
\begin{matrix}
\begin{aligned}
\dfrac{\partial^2 \phi }{\partial x^2} + \dfrac{\partial^2 \phi}{\partial y^2} + \dfrac{\partial^2 \phi }{\partial z^2}= -\ff (x,y,z), \quad &\text{in} \quad \Omega=(0,1)^3,\\
\phi=0 \quad &\text{in} \quad \partial \Omega,
\end{aligned}
\end{matrix}
\right.
\end{equation}
where $\ff(x,y,z)=3 \cdot (2\pi)^2 \cdot \sin(2\pi x-\pi)\sin(2\pi y-\pi)\sin(2\pi z-\pi)$. This problem has a closed form solution
$$
\phi(x,y,z)=\sin(2\pi x-\pi)\sin(2\pi y-\pi)\sin(2\pi z-\pi).
$$
By using derivative approximations and finite difference methods, we can write the Poisson equation in discrete form as a linear system $A \cdot \phi_{ijk} = -\ff_{ijk}$, where the indices $i,j,k$ correspond to the discretization of $x,y$ and $z$ respectively, and $A$ is a block matrix (see \cite{greedy} for more details). In Figure \ref{poisson}, we compare the CPU time employed in solving this discrete Poisson problem using the GROU Algorithm and the Matlab operator $\xx=A$\textbackslash $\bb$, for different numbers of nodes in $(0,1)^3$.
\begin{figure}[h]
	\centering
	\includegraphics[width=0.8\columnwidth]{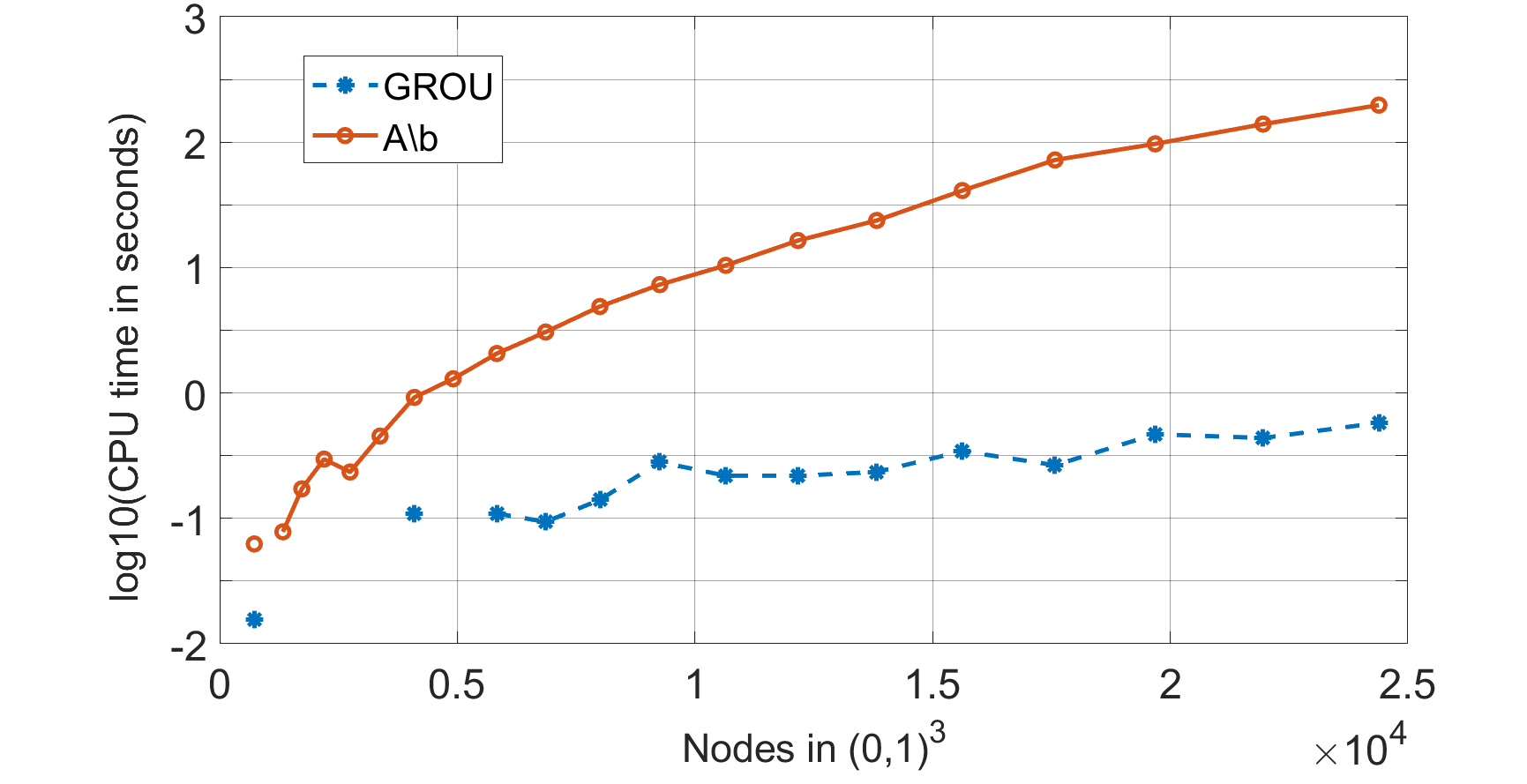} %0.75
	\caption{CPU time comparative to solve the discrete Poisson equation. For this numerical test, we have used a computer with the following characteristics: 11th Gen Intel(R) Core(TM) i7-11370H @ 3.30GHz, RAM 16,0 GB, 64 bit operating system; and a Matlab version R2021b \cite{matlab}.}
	\label{poisson}
\end{figure}

So, we will use this fact to study if, for a given generic square matrix, a characterization can be stated such that we can decide whether is either Laplacian-like or not. Clearly, under a positive answer, we expect that the analysis of the associated linear system $A\mathbf{x}=\mathbf{b}$ would be simpler. 
This kind of linear operator also exists in infinite dimensional vector spaces to describe evolution equations in tensor Banach spaces \cite{Falcob}. Its main property is that the associated dynamical system has an invariant manifold, the manifold of elementary tensors (see \cite{Falc2018} for the details about its manifold structure).\medskip

Thus, the goal of this paper is to obtain a complete description of this linear space of matrices, showing that is, in fact, a Lie subalgebra of $\mathbb{R}^{N\times N}$, and provide an algorithm in order to obtain the best approximation to this linear space, that is, to compute explicitly is the orthogonal projection on that space.\medskip

The paper is organized as follows: in Section \ref{sectionGROA}, we introduce the linear subspace of Laplacian-like matrices and prove that it is also a matrix Lie sub-algebra associated to a particular Lie group. %application to different linear systems. 
Then, in Section \ref{Lap_results}, we prove that 
any matrix is uniquely decomposed as the sum of a Laplacian matrix and a matrix which is the subspace generated by the identity matrix, and we show that any Laplacian matrix is a direct sum of some particular orthogonal subspaces.  Section \ref{results} is devoted, with the help of the results of the previous section, to propose an algorithm to explicitly compute the orthogonal projection onto the subspace of Laplacian-like matrices. To illustrate this result, we also give some numerical examples. 
Finally, in Section \ref{conclusions} some conclusions and final remarks are given. 

%%%%%%%%%%%%%%%%%%%%%%%%%%%%%%%%%%%%%%%%%%%%%%%%%%%%%%%%%

\section{The algebraic structure of Laplacian-Like matrices}\label{sectionGROA}

First of all, we introduce some definitions, that will be used along this work.

\medskip

\begin{definition}
Let $A \in \R^{M \times N}.$ Then, the Fröbenius norm (or the Hilbert–Schmidt norm) is defined as
$$
||A||_F=\sqrt{\sum_{i=1}^M \sum_{j=1}^N |a_{ij}|^2} = \sqrt{\tr\left(A^{\top}A\right)}.
$$
\end{definition}

The Fröbenius norm is the norm induced by the trace therefore, when $N=M$, we can work with 
the scalar product given by $\langle A,B\rangle=\tr\left(A^{\top}B\right)$. Let us observe that, in $\R^{N \times N}$,
\begin{enumerate}
    \item $\langle A,B\rangle_{\R^{N\times N}}=\tr\left(A^{\top}B\right)$
    \item $\langle A,\id_{N}\rangle_{\R^{N\times N}}=\tr(A)=\tr\left(A^{\top}\right)$
    \item $\langle \id_{N},\id_{N}\rangle_{\R^{N\times N}}=||\id_{N}||_F^2=N.$
\end{enumerate}

Given a linear subspace $\mathcal U \subset \mathbb{R}^{N \times N}$ we will denote: 
\begin{enumerate}
\item[(a)] the orthogonal complement
of $\mathcal U$ in $\mathbb{R}^{N \times N}$ by
$$
\mathcal U^{\bot} = \left\{ V \in \mathbb{R}^{N \times N}: \langle U , V\rangle_{\mathbb{R}^{N \times N}} = 0 \text{ for all } U \in \mathcal U \right\},
$$
and,
\item[(b)] the orthogonal projection of $ \R^{N \times N}$ on $\mathcal U$ as
$$
P_{\mathcal U}(V) := \arg \min_{U \in \mathcal U}\|U-V\|_F,
$$
and hence
$$
P_{\mathcal U^{\bot}} = \id_N - P_{\mathcal U}.
$$
\end{enumerate}

Before defining a Laplacian-like matrix, we recall that the \emph{Kronecker product} of two matrices $A \in \R^{N_1 \times M_1}$, $B \in \R^{N_2 \times M_2}$ is defined by
\begin{equation*}
    A \otimes B = 
    \begin{pmatrix}
    A_{1,1}B & A_{1,2}B & \dots & A_{1,M_1}B \\
    A_{2,1}B & A_{2,2}B & \dots & A_{2,M_1}B \\
    \vdots & \vdots & \ddots & \vdots \\
    A_{N_1,1}B & A_{N_1,2}B & \dots & A_{N_1,M_1}B \\
    \end{pmatrix} \in \R^{N_1N_2 \times M_1M_2}.
\end{equation*}
Some of the well-known properties of the Kronecker product are:
\begin{enumerate}
    \item $A \otimes (B \otimes C) = (A \otimes B) \otimes C$.
    \item $(A + B) \otimes C = (A \otimes C)+(B \otimes C)$.
    \item $AB \otimes CD = (A \otimes C)(B \otimes D)$.
    \item $(A \otimes B)^{-1}=A^{-1}\otimes B^{-1}$.
    \item $(A \otimes B)^{\top} = A^{\top} \otimes B^{\top}$.
    \item $\tr(A \otimes B) = \tr(A) \tr(B).$
\end{enumerate}

From the example given in the introduction, 
we observe that there is a particular type of matrices to solve high-dimensional linear systems for which the GROU algorithm works particularly well:  very fast convergence and also a very good approximation of the solution. These are the so-called Laplacian-Like matrices that we define below.

\medskip

\begin{definition} Given a matrix $A \in \mathbb{R}^{N \times N},$ where $N=n_1\cdots n_d,$ we say that
$A$ is a Laplacian-like matrix if there exist matrices $A_i \in \R^{n_i \times n_i}$ for $1 \le i \le d$
be such that
\begin{equation}\label{laplaciana}
A=\sum_{i=1}^d \mathrm{id}_{[n_i]} \otimes A_{i} \doteq \sum_{i=1}^d \mathrm{id}_{n_1} \otimes \dots \otimes \mathrm{id}_{n_{i-1}} \otimes A_{i} \otimes \mathrm{id}_{n_{i+1}} \otimes \dots \otimes \mathrm{id}_{n_d},
\end{equation}
where $\id_{n_j}$ is the identity matrix of size $n_j \times n_j.$ 
\end{definition}

It is not difficult to see that the set of Laplacian-like matrices is a linear subspace of $\mathbb{R}^{N \times N}$. From now on, we will denote  by $\mathcal{L}\left(\R^{N \times N}\right)$ the subspace of Laplacian-like matrices in $\R^{N \times N}$ for a fixed decomposition of $N=n_1\cdots n_d$.\medskip

These matrices can be easily related to the classical Laplacian operator \cite{laplace1,laplace2} by writing:
$$
\frac{\partial^2}{\partial x_i^2} = \frac{\partial^0}{\partial x_1^0} \otimes \dots \otimes \frac{\partial^0}{\partial x_{i-1}^0} \otimes \frac{\partial^2}{\partial x_i^2} \otimes\frac{\partial^0}{\partial x_{i+1}^0} \otimes \dots \otimes \frac{\partial^0}{\partial x_d^0}
$$
and where $\frac{\partial^0}{\partial x_j^0}$ is the identity operator for functions in the variable $x_j$ for $j\neq i$.\medskip

As the next numerical example shows, matrices written as in \eqref{laplaciana}
provides very good performance of the GROU algorithm.  In Figure \ref{ll_vs_full} we give a comparison of the speed of convergence to solve a linear system $A \mathbf{x} = \mathbf{b}$, where for each fixed size, we randomly generated two full-rank matrices: one given in the classical form and a Laplacian-like matrix. Both systems were solved following Algorithm~\ref{GROU_alg}.\medskip

\begin{figure}[h]
	\centering
	\includegraphics[width=0.83\columnwidth]{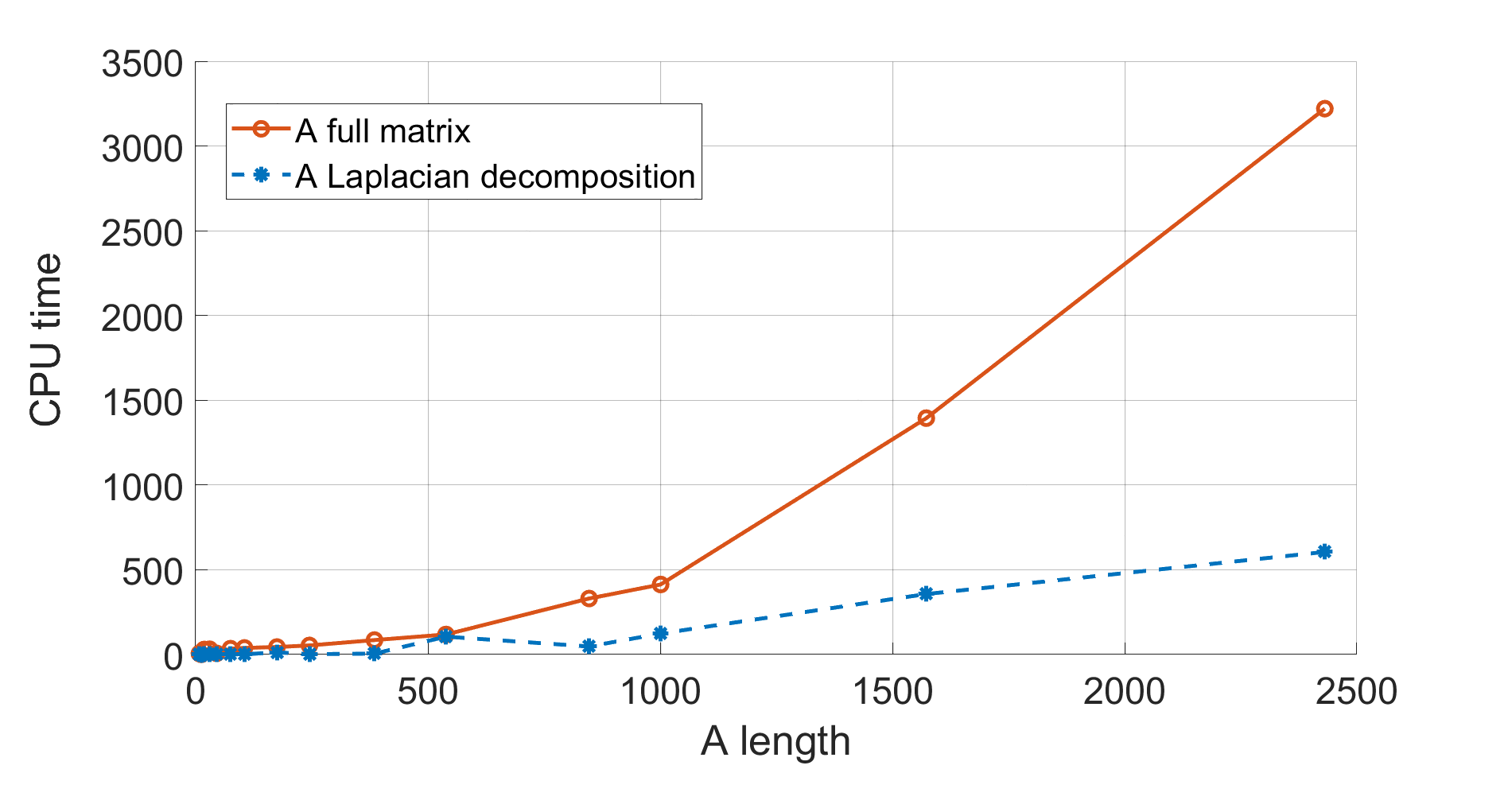} %0.75
	\caption{CPU time comparative to solve an $A\xx=\bb$ problem. This graph has been generated by using the following data in Algorithm~\ref{GROU_alg}: $\texttt{tol} = 2.22e-6$; $\varepsilon = 1.0e-06$;  $\texttt{rank\_max} = 3000$; (an $\texttt{iter-max}=15$ was used to perform an ALS strategy); and the matrices have been randomly generated for each different size, in Laplacian and classical form. The characteristics of the computer used here are the same 
	as in the case of Figure 1.} %Full matrix VS Laplacian-Like matrix}
	\label{ll_vs_full}
\end{figure}

The above results, together with the previous Poisson example given in the introduction, motivate the interest to know for a given matrix $A \in \mathbb{R}^{N \times N}$ how far it is from the linear subspace of Laplacian-like matrices. More precisely, we are interested in decomposing any matrix $A$ as a sum of two orthogonal matrices $L$ and $L^{\bot},$ where $L$ is in $\mathcal{L}(\mathbb{R})$ and $L^{\bot}$ in $\mathcal{L}(\mathbb{R})^{\bot}.$ Clearly, if we obtain that $L^{\bot} =0,$ that is, $A \in \mathcal{L}(\mathbb{R}),$ then we can solve any associated linear system by means of the GROU algorithm.\medskip

Recall that the set of matrices $\mathbb R^{N \times N}$ is a Lie Algebra that appears as the tangent space at the identity matrix of the linear general group $GL(\mathbb{R}^N),$ a Lie group composed by the non-singular matrices of $\mathbb{R}^{N \times N}$ (see \cite{Gallier2020}). Furthermore, the exponential map
$$
\exp: \mathbb R^{N \times N} \longrightarrow GL(\mathbb{R}^{N}), \quad A \mapsto \exp(A)= \sum_{n=0}^{\infty} \frac{A^n}{n!}
$$
is well-defined, however it is not surjective because $\det(\exp(A)) = e^{\tr(A)} > 0.$ Any linear subspace $\mathfrak{h} \subset \mathbb R^{N \times N}$ is a Lie-subalgebra if for all $A,B \in \mathfrak{h}$ its Lie crochet is also in $\mathfrak{h},$ that is, $[A,B]=AB-BA \in \mathfrak{h}$.
\medskip

The linear space $\mathcal{L}(\mathbb R^{N \times N})$ is more than a linear subspace of $\mathbb R^{N \times N},$ it is also a Lie sub-algebra of $\mathbb R^{N \times N}$ as the next result shows.

\begin{proposition}\label{Lie} Assume $\mathbb{R}^{N \times N},$ where where $N=n_1\cdots n_d.$ Then the following statements hold.
\begin{enumerate}
\item[(a)]The linear subspace $\mathcal{L}(\mathbb R^{N \times N})$ is a Lie subalgebra of the matrix Lie algebra $R^{N \times N}.$ 
\item[(b)] The matrix group
$$
\mathfrak{L}(\mathbb R^{N \times N})=\left\{
\bigotimes_{i=1}^d A_i: A_i \in GL(\mathbb R^{n_i}) \text{ for } 1 \le i \le d
\right\}
$$
is a Lie subgroup of $GL\left(\mathbb{R}^N\right).$
\item[(c)] The exponential map
$$
\exp:\mathcal{L}(\mathbb R^{N \times N}) \longrightarrow \mathfrak{L}(\mathbb R^{N \times N}), 
\quad \sum_{i=1}^d \mathrm{id}_{n_1} \otimes \dots \otimes \mathrm{id}_{n_{i-1}} \otimes A_{i} \otimes \mathrm{id}_{n_{i+1}} \otimes \dots \otimes \mathrm{id}_{n_d} \mapsto \bigotimes_{i=1}^d \exp(A_i),
$$
is well defined.
\end{enumerate}
\end{proposition}

\begin{proof}
(a) To prove the first statement, take $A,B \in \mathcal{L}(\mathbb R^{N \times N}).$ Then
there exist matrices $A_i,B_i \in \R^{n_i \times n_i}$ for $1 \le i \le d$
be such that
$$
A=\sum_{i=1}^d \mathrm{id}_{n_1} \otimes \dots \otimes \mathrm{id}_{n_{i-1}} \otimes A_{i} \otimes \mathrm{id}_{n_{i+1}} \otimes \dots \otimes \mathrm{id}_{n_d},
$$
and
$$
B= \sum_{i=1}^d \mathrm{id}_{n_1} \otimes \dots \otimes \mathrm{id}_{n_{i-1}} \otimes B_{i} \otimes \mathrm{id}_{n_{i+1}} \otimes \dots \otimes \mathrm{id}_{n_d}.
$$

Observe, that for $i < j$
$$
\left(\mathrm{id}_{n_1} \otimes \dots \otimes \mathrm{id}_{n_{i-1}} \otimes A_{i} \otimes \mathrm{id}_{n_{i+1}} \otimes \dots \otimes \mathrm{id}_{n_d}\right) \left(\mathrm{id}_{n_1} \otimes \dots \otimes \mathrm{id}_{n_{j-1}} \otimes B_{j} \otimes \mathrm{id}_{n_{j+1}} \otimes \dots \otimes \mathrm{id}_{n_d}\right)
$$
and
$$
\left(\mathrm{id}_{n_1} \otimes \dots \otimes \mathrm{id}_{n_{j-1}} \otimes B_{j} \otimes \mathrm{id}_{n_{j+1}} \otimes \dots \otimes \mathrm{id}_{n_d}\right)\left(\mathrm{id}_{n_1} \otimes \dots \otimes \mathrm{id}_{n_{i-1}} \otimes A_{i} \otimes \mathrm{id}_{n_{i+1}} \otimes \dots \otimes \mathrm{id}_{n_d}\right)
$$
both products are equal to
$$
\mathrm{id}_{n_1} \otimes \dots \otimes \mathrm{id}_{n_{i-1}} \otimes A_{i} \otimes \mathrm{id}_{n_{i+1}} \otimes \cdots \otimes \mathrm{id}_{n_{j-1}} \otimes B_{j} \otimes \mathrm{id}_{n_{j+1}} \otimes \dots \mathrm{id}_{n_d}. 
$$
A similar expression is obtained for $i > j.$  Thus, 
$$
\left[\mathrm{id}_{n_1} \otimes \dots \otimes \mathrm{id}_{n_{i-1}} \otimes A_{i} \otimes \mathrm{id}_{n_{i+1}} \otimes \dots \otimes \mathrm{id}_{n_d},\mathrm{id}_{n_1} \otimes \dots \otimes \mathrm{id}_{n_{j-1}} \otimes B_{j} \otimes \mathrm{id}_{n_{j+1}} \otimes \dots \otimes \mathrm{id}_{n_d}\right] = 0
$$
for all $i \neq j.$
\medskip

On the other hand, for $i=j$ we have
$$
\left(\mathrm{id}_{n_1} \otimes \dots \otimes \mathrm{id}_{n_{i-1}} \otimes A_{i} \otimes \mathrm{id}_{n_{i+1}} \otimes \dots \otimes \mathrm{id}_{n_d}\right) \left(\mathrm{id}_{n_1} \otimes \dots \otimes \mathrm{id}_{n_{i-1}} \otimes B_{i} \otimes \mathrm{id}_{n_{i+1}} \otimes \dots \otimes \mathrm{id}_{n_d}\right)
$$
is equal to
$$
\mathrm{id}_{n_1} \otimes \dots \otimes \mathrm{id}_{n_{i-1}} \otimes A_{i}B_{i} \otimes \mathrm{id}_{n_{i+1}} \otimes \dots \otimes \mathrm{id}_{n_d}
$$
and
$$
\left(\mathrm{id}_{n_1} \otimes \dots \otimes \mathrm{id}_{n_{i-1}} \otimes B_{i} \otimes \mathrm{id}_{n_{i+1}} \otimes \dots \otimes \mathrm{id}_{n_d}\right) \left(\mathrm{id}_{n_1} \otimes \dots \otimes \mathrm{id}_{n_{i-1}} \otimes A_{i} \otimes \mathrm{id}_{n_{i+1}} \otimes \dots \otimes \mathrm{id}_{n_d}\right)
$$
is equal to
$$
\mathrm{id}_{n_1} \otimes \dots \otimes \mathrm{id}_{n_{i-1}} \otimes B_{i}A_{i} \otimes \mathrm{id}_{n_{i+1}} \otimes \dots \otimes \mathrm{id}_{n_d}.
$$
Thus,
$$
\left[\mathrm{id}_{n_1} \otimes \dots \otimes \mathrm{id}_{n_{i-1}} \otimes A_{i} \otimes \mathrm{id}_{n_{i+1}} \otimes \dots \otimes \mathrm{id}_{n_d},\mathrm{id}_{n_1} \otimes \dots \otimes \mathrm{id}_{n_{i-1}} \otimes B_{i} \otimes \mathrm{id}_{n_{i+1}} \otimes \dots \otimes \mathrm{id}_{n_d}\right]
$$
is equal to
$$
\mathrm{id}_{n_1} \otimes \dots \otimes \mathrm{id}_{n_{i-1}} \otimes (A_iB_i - B_{i}A_{i}) \otimes \mathrm{id}_{n_{i+1}} \otimes \dots \otimes \mathrm{id}_{n_d},
$$
that is,
$$
\mathrm{id}_{n_1} \otimes \dots \otimes \mathrm{id}_{n_{i-1}} \otimes [A_i,B_i] \otimes \mathrm{id}_{n_{i+1}} \otimes \dots \otimes \mathrm{id}_{n_d}.
$$
Here $[A_i,B_i]$ is the Lie crochet in $\mathbb{R}^{n_i \times n_i}.$

\medskip

In consequence, from all said above, we conclude
\begin{align*}
[A,B] & = \sum_{i=1}^d\sum_{j=1}^d\left[ \mathrm{id}_{n_1} \otimes \dots \otimes \mathrm{id}_{n_{i-1}} \otimes A_{i} \otimes \mathrm{id}_{n_{i+1}} \otimes \dots \otimes \mathrm{id}_{n_d},\mathrm{id}_{n_1} \otimes \dots \otimes \mathrm{id}_{n_{i-1}} \otimes B_{j} \otimes \mathrm{id}_{n_{i+1}} \otimes \dots \otimes \mathrm{id}_{n_d}\right] \\
& = \sum_{i=1}^d\mathrm{id}_{n_1} \otimes \dots \otimes \mathrm{id}_{n_{i-1}} \otimes [A_i,B_i] \otimes \mathrm{id}_{n_{i+1}} \otimes \dots \otimes \mathrm{id}_{n_d} \in \mathcal{L}\left(\mathbb R^{N \times N}\right).
\end{align*}
This proves that $\mathcal{L}(\mathbb R^{N \times N})$ is a 
Lie sub-algebra of $\mathbb R^{N \times N}.$ 
\medskip

(b) It is not difficult to see that $\mathfrak{L}(\mathbb{R}^{N\times N})$ is a subgroup 
of $GL(\mathbb R^{N}).$ From Theorem~19.18 in \cite{Gallier2020}, 
to prove that $\mathfrak{L}(\mathbb{R}^{N\times N})$ is a 
Lie subgroup of $GL(\mathbb{R}^N)$ we only need to show that $\mathfrak{L}(\mathbb{R}^{N\times N})$ 
is a closed set in $GL(\mathbb{R}^N).$ This follows from the fact that the map
$$
\Phi:GL(\mathbb{R}^{n_1}) \times \cdots \times GL(\mathbb{R}^{n_d}) \longrightarrow
GL(\mathbb{R}^N) \quad (A_1,\cdots,A_d) \mapsto \bigotimes_{i=1}^d A_i
$$
is continuous. Assume that there exists a sequence, $\{A_n\}_{n \in \mathbb{N}} \subset \mathfrak{L}(\mathbb{R}^{N\times N})$ convergent to $A \in GL(\mathbb{R}^n).$ Then the sequence
$\{A_n\}_{n \in \mathbb{N}}$ is bounded. Since there exists a sequence
$\{(A_1^{(n)},\ldots,A_d^{(n)})\}_{n \in \mathbb{N}} \subset GL(\mathbb{R}^{n_1}) \times \cdots \times GL(\mathbb{R}^{n_d})$ such that $A_n = \bigotimes_{j=1}^d A_j^{(n)},$ the sequence
$\{(A_1^{(n)},\ldots,A_d^{(n)})\}_{n \in \mathbb{N}}$ is also bounded. Thus, there exists a convergent sub-sequence, also denoted by $\{(A_1^{(n)},\ldots,A_d^{(n)})\}_{n \in \mathbb{N}},$ to $(A_1,\ldots,A_d) \in GL(\mathbb{R}^{n_1}) \times \cdots \times GL(\mathbb{R}^{n_d}).$
The continuity of $\Phi,$ implies that $A =  \bigotimes_{i=1}^d A_i.$ Thus $\mathfrak{L}(\mathbb{R}^{N\times N})$ is closed in $GL(\mathbb{R}^N),$ and hence a Lie subgroup.

(c) From Lemma 4.169(b)\cite{graham}, the following equality
$$
\exp\left( \sum_{i=1}^d \mathrm{id}_{n_1} \otimes \dots \otimes \mathrm{id}_{n_{i-1}} \otimes A_{i} \otimes \mathrm{id}_{n_{i+1}} \otimes \dots \otimes \mathrm{id}_{n_d}\right) = \bigotimes_{i=1}^d \exp(A_i)
$$
holds. Thus, the exponential map is well defined. This ends the proof of the proposition.
\end{proof}

We conclude this section describing in a more detail the structure of matrices $A\in \R^{N \times N}$ for which there exists $A_i \in \R^{n_i \times n_i}$ for $1 \le i \le d$ such that
$$
A = \sum_{i=1}^d \id_{n_1} \otimes \dots \id_{n_{i-1}} \otimes A_i \otimes \id_{n_{i+1}} \otimes \dots \otimes \id_{n_d}.
$$

For dealing easily with Laplacian-like matrices, we introduce the following notation. For each $1 < i \le d$ consider the integer number
$n_1n_2\cdots n_{i-1}.$ Then, we will denote by
$$
\left(
\begin{array}{cccc}
\star & \star & \cdots &\star \\ 
\star & \star & \cdots &\star \\ 
\vdots & \vdots & \ddots & \vdots \\ 
\star & \star & \cdots & \star
\end{array}
\right)_{n_1n_2\cdots n_{i-1} \times n_1n_2\cdots n_{i-1}}
$$
a block square matrix composed by $n_1n_2\cdots n_{i-1} \times n_1n_2\cdots n_{i-1}$-blocks.
Then, we observe, that for $1 < i < d,$ we can write
$$
\id_{n_1} \otimes \dots \id_{n_{i-1}} \otimes A_i \otimes \id_{n_{i+1}} \otimes \dots \otimes \id_{n_d} =\id_{n_1\cdots n_{i-1}} \otimes A_i \otimes \id_{n_{i+1}\cdots n_d}.
$$

Since
$$
A_i \otimes \id_{n_{i+1}\cdots n_d} =
 \begin{pmatrix}
    (A_i)_{1,1}\id_{n_{i+1}\cdots n_d} \phantom{a}& (A_i)_{1,2}\id_{n_{i+1}\cdots n_d} \phantom{a}& \dots \phantom{a}& (A_i)_{1,n_i}\id_{n_{i+1}\cdots n_d} \\
    (A_i)_{2,1}\id_{n_{i+1}\cdots n_d} \phantom{a}& (A_i)_{2,2}\id_{n_{i+1}\cdots n_d} \phantom{a}& \dots \phantom{a}& (A_i)_{2,n_i}\id_{n_{i+1}\cdots n_d} \\
    \vdots \phantom{a}& \vdots \phantom{a}& \ddots \phantom{a}& \vdots \\
    (A_i)_{n_i,1}\id_{n_{i+1}\cdots n_d} \phantom{a}& (A_i)_{n_i,2}\id_{n_{i+1}\cdots n_d} \phantom{a}& \dots \phantom{a}& (A_i)_{n_i,n_i}\id_{n_{i+1}\cdots n_d} \\
    \end{pmatrix},
$$
\medskip
\noindent 
then 
$$
\phantom{a}\id_{n_1\cdots n_{i-1}} \otimes A_i \otimes \id_{n_{i+1}\cdots n_d}
= \begin{pmatrix}
A_i \otimes \id_{n_{i+1}\cdots n_d} \phantom{a}& O_i \otimes \id_{n_{i+1}\cdots n_d} \phantom{a}& \cdots \phantom{a}&O_i \otimes \id_{n_{i+1}\cdots n_d} \\ 
O_i \otimes \id_{n_{i+1}\cdots n_d} \phantom{a}& A_i \otimes \id_{n_{i+1}\cdots n_d}\phantom{a}& \cdots \phantom{a}&O_i \otimes \id_{n_{i+1}\cdots n_d}\\ 
\vdots \phantom{a}& \vdots \phantom{a}& \ddots \phantom{a}& \vdots \\ 
O_i \otimes \id_{n_{i+1}\cdots n_d} \phantom{a}& O_i \otimes \id_{n_{i+1}\cdots n_d} \phantom{a}& \cdots \phantom{a}&A_i \otimes \id_{n_{i+1}\cdots n_d}
\end{pmatrix}_{n_1n_2\cdots n_{i-1} \times n_1n_2\cdots n_{i-1}},
$$
\medskip
\noindent 
where $O_i$ denotes the zero matrix in $\R^{n_i \times n_i}$ for $1 \le i \le d.$ 
To conclude, we have the following cases
$$
A_1 \otimes \id_{n_{2}\cdots n_d} =
\begin{pmatrix}
    (A_1)_{1,1}\id_{n_{2}\cdots n_d} \phantom{a}& (A_1)_{1,2}\id_{n_{2}\cdots n_d} \phantom{a}& \dots & (A_1)_{1,n_1}\id_{n_{2}\cdots n_d} \\
    (A_1)_{2,1}\id_{n_{2}\cdots n_d} \phantom{a}& (A_1)_{2,2}\id_{n_{2}\cdots n_d} \phantom{a}& \dots \phantom{a}& (A_1)_{2,n_1}\id_{n_{2}\cdots n_d} \\
    \vdots \phantom{a}& \vdots \phantom{a}& \ddots \phantom{a} & \vdots \\
    (A_1)_{n_1,1}\id_{n_{2}\cdots n_d} \phantom{a}& (A_1)_{n_1,2}\id_{n_{2}\cdots n_d} \phantom{a}& \dots \phantom{a}& (A_1)_{n_1,n_1}\id_{n_{2}\cdots n_d} \\
    \end{pmatrix}
$$\medskip
\noindent and
$$
\id_{n_1\cdots n_{d-1}} \otimes A_d =
\begin{pmatrix}
A_d  & O_d  & \cdots &O_d  \\ 
O_d  & A_d & \cdots &O_d \\ 
\vdots & \vdots & \ddots & \vdots \\ 
O_d  & O_d  & \cdots &A_d 
\end{pmatrix}_{n_1n_2\cdots n_{d-1} \times n_1n_2\cdots n_{d-1}}.
$$
\medskip

We wish to point out that the above operations are widely used in quantum computing.

%%%%%%%%%%%%%%%%%%%%%%%%%%%%%%%%%%%%%%%
%               SECCION 2
%%%%%%%%%%%%%%%%%%%%%%%%%%%%%%%%%%%%%%%
\section{A decomposition of the linear space of Laplacian-like matrices}\label{Lap_results}

We start by introducing some definitions and preliminary results needed to give 
an interesting decomposition of the linear space of Laplacian-like matrices. The next lemma lets us show how is the decomposition of $\R^{N \times N}$ as a direct sum of $\spn\{\id_{N}\}$ and its orthogonal space.\medskip

\begin{lemma}\label{decomposition}
Consider $\left(\R^{N \times N}, \|\cdot\|_F\right)$ as a Hilbert space. Then there exists a decomposition
$$
\R^{N \times N} = \spn\{\id_{N}\} \oplus \spn\{\id_{N}\}^{\bot},
$$  
where $\spn\{\id_{N}\}^{\bot} = \{ A \in \R^{N \times N}: \tr(A) = 0\}.$
Moreover, the orthogonal projection from $ \R^{N \times N}$
on $\spn\{\id_{N}\}$ is given by
$$
P_{\spn\{\id_{N}\}}(A) =  \frac{\tr(A)}{N} \, \id_{N},
$$ 
and hence for each $A \in \R^{N \times N}$ we have the following decomposition,
$$
A= \frac{\tr(A)}{N} \, \id_{N} + \left( A - \frac{\tr(A)}{N} \id_{N}\right),
$$
where $\left( A - \frac{\tr(A)}{N} \id_{N}\right) \in  \spn\{\id_{N}\}^{\bot}.$
\end{lemma}

\begin{proof}

The lemma follows from the fact that
$$
P_{\id_N}(A) = \frac{\langle \id_N,A\rangle_{\R^{N\times N}}}{\|\id_N\|_F^2} \, \id_n  = \frac{\tr(A)}{N} \, \id_{N},
$$
is the orthogonal projection onto $\spn\{\id_{N}\}.$
\end{proof}

Now, we consider the matrix space $\mathbb{R}^{N \times N}$ where $N=n_1\cdots n_d,$ and hence
$\mathbb{R}^{N \times N} = \bigotimes_{i=1}^d \R^{n_i \times n_i}$ can be considered as a tensor space. 
Then, for rank-one tensors $A=A_1 \otimes \dots \otimes A_d$ and $B=B_1 \otimes \dots \otimes B_d$ where $A_i,B_i \in \R^{n_i \times n_i},$
we have
\begin{align*}
\langle A, B \rangle_{\R^{N \times N}} & = \langle A_1 \otimes \dots \otimes A_d, B_1 \otimes \dots \otimes B_d \rangle_{\R^{N \times N}} = \tr((A_1 \otimes \dots \otimes A_d)^T (B_1 \otimes \dots \otimes B_d)) \\ 
& =  \tr((A_1^T \otimes \dots \otimes A_d^T) (B_1 \otimes \dots \otimes B_d)) =  \tr(A_1^TB_1 \otimes \dots \otimes A_d^TB_d) \\
& = \prod_{i=1}^d \tr(A_i^{\top}B_i) = \prod_{i=1}^d \langle A_i,B_i\rangle_{\R^{n_i\times n_i}} .
\end{align*}
Thus, the inner product $\langle \cdot, \cdot \rangle_{\mathbb R^{N \times N}}$ satisfies
\begin{equation}\label{eq1}
\langle \id_{[n_i]} \otimes A_i, \id_{[n_i]} \otimes B_i \rangle_{\R^{N \times N}} = \tr(A_i^{\top}B_i) \prod_{\substack{j=1\\j \neq i}}^d n_j,
\end{equation}
\medskip
and $\|A\|_F= \sqrt{\langle A, A \rangle_{\R^{N \times N}}},$ is called a tensor-norm.

%\textcolor{red}{
%We are looking for two subspaces $V, W$ such as $\R^{N \times N}=V \oplus W$, that is, we want decompose $\R^{N \times N}$ into the direct sum of two subspaces, where one of them has Laplacian form. In that point, we can project any matrix in that subspace and obtain its Laplacian approximation. %
% 
%In order to look for that subspace, we do the following study/ remarks:}\\

The next result gives a first characterization of the 
linear space $\mathcal L\left(\mathbb R^{N \times N}\right).$ 

\medskip

\begin{theorem}\label{caract1Lap} Let $\R^{N \times N},$ where $N=n_1\cdots n_d.$ Then
\begin{equation}
\mathcal{L}\left(\mathbb{R}^{N \times N}\right)= \mathrm{span}\, \{ \id_{N}\} \oplus \Delta,
\end{equation}
where $\Delta = (\spn\{ \id_{N}\}^{\bot} \cap \mathcal{L}(\R^{N \times N})).$ Furthermore,
$\mathcal{L}(\R^{N \times N})^{\bot}$ is a subspace of $\spn\{ \id_{N}\}^{\bot}.$
\end{theorem}

%\textcolor{blue}{Ahora, dada una matriz cuadrada cualquiera en $\R^{N \times N},$ nos gustaría proyectarla sobre $\mathcal{L}$ para obtener su aproximación laplaciana. Sin embargo, calcular dicha proyeccion es dificil, debido a que los subconjuntos $\R^{n_i\times n_i} \otimes \id_{[n_i]}$, $i=1,\dots,d$ no son disjuntos dos a dos (i. e., for different $i,j$, la interseccion es no vacía); en particular, la matriz $\id_N$ está en cada uno de ellos.}\\
%***********
\begin{proof}
Assume that a given matrix $A \in \mathbb{R}^{N \times N}$ can be written as in \eqref{laplaciana}. 
and denote each component in the sum representation of $A,$ by $L_i=\id_{[n_i]} \otimes A_i,$ where $A_i \in \R^{n_i \times n_i}$ for $1 \le i \le d.$ Then $L_i \in \spn\{\id_{[n_i]}\} \otimes \R^{n_i\times n_i}$ for $1 \le i \le d,$ and in consequence,
$$
\sum_{i=1}^d \spn\{\id_{[n_i]}\} \otimes \R^{n_i\times n_i} = \mathcal{L}\left(\mathbb{R}^{N \times N}\right).
$$\medskip

Thus, $\spn\{ \id_{N}\} \subset  \mathcal{L}\left(\mathbb{R}^{N \times N}\right),$ and, by Lemma~\ref{decomposition}, we have the following decomposition
\begin{equation}\label{eq2thdelta}
\mathcal{L}\left(\mathbb{R}^{N \times N}\right) = \Delta \oplus \spn\{\id_N\}.
\end{equation}
where $\Delta = \left(\spn\{ \id_{N}\}^{\bot} \cap \mathcal{L}(\R^{N \times N})\right).$ The last statement is consequence of Lemma~\ref{decomposition}. This ends the theorem.
\end{proof}

Now, given any square matrix in $\R^{N \times N},$ we would like to 
project it onto $\mathcal{L}(\mathbb{R}^{N \times N})$ to obtain its Laplacian approximation. 
To compute this approximation explicitly, the following result, which is a consequence of the above theorem, will be useful.

\begin{corollary}\label{optimal_decom}
Assume $\mathbb{R}^{N \times N},$ with $N=n_1 \cdots n_d \in \N.$ Then
$$
P_{\mathcal{L}(\R^{N \times N})} = P_{\spn\{\id_N\}} + P_{\Delta},
$$
that is, for all $A \in \R^{N \times N}$ it holds
$$
P_{\mathcal{L}(\R^{N \times N})}(A) = \frac{\tr(A)}{N}\, \id_N + P_{\Delta}(A).
$$
\end{corollary}

Next, we need to characterize $\Delta$ in order to explicitly construct the orthogonal projection
onto $\mathcal L(\mathbb R^{N \times N})$. From the proof of the Theorem~\ref{caract1Lap} we see that the linear subspaces given
by
$$
 \spn\{ \id_{N}\}^{\bot} \cap  \spn\{ \id_{[n_i]}\} \otimes \R^{n_i\times n_i},
$$
for $1 \le i \le d,$ are of interest to characterize $\Delta$ as the next result shows.

\begin{theorem}\label{thdelta}%[$\R^{N \times N}$ decomposition]
Let $\mathbb{R}^{N \times N}$ with $N=n_1 \cdots n_d \in \N,$ and let $\spn\{ \id_{N} \}^{\perp_i}$ be the orthogonal complement of $\spn\{ \id_{N} \}$ in the linear subspace $\spn\{\id_{[n_i]}\} \otimes \R^{n_i \times n_i}$ for $1 \leq i \leq d$. Then,
\begin{equation}\label{eq1thdelta}
\Delta=\bigoplus_{i=1}^d  \spn\{ \id_{N} \}^{\perp_i}.
\end{equation}
Furthermore, a matrix $A$ belongs to $\Delta$ if and only if it has the form
$$
A=\sum_{i=1}^d \id_{n_1} \otimes \dots \id_{n_{i-1}} \otimes A_i \otimes \id_{n_{i+1}} \otimes \dots \otimes \id_{n_d}, \quad \text{with} \quad \tr(A_i)=0, \; i=1,\dots,d.
$$
\end{theorem}

\begin{proof}
First, we take into account that $\spn\{ \id_{[n_i]}\} \otimes \R^{n_i\times n_i}$ a linear subspace of $\mathcal{L}(\mathbb{R}^{N \times N})$ linearly isomorphic to the
matrix space $\R^{n_i\times n_i}.$ Thus, motivated by Lemma~\ref{decomposition} applied on $\mathbb{R}^{n_i \times n_i}$, we write
\begin{align*}
\spn\{ \id_{[n_i]}\} \otimes \R^{n_i\times n_i} & = \spn\{ \id_{[n_i]}\} \oplus (\spn\{\id_{n_i}\} \otimes \spn\{\id_{n_i}\}^{\bot}) \\ 
& = (\spn\{ \id_{[n_i]}\} \otimes \spn\{\id_{n_i}\}) \oplus (\spn\{ \id_{[n_i]}\} \otimes \spn\{\id_{n_i}\}^{\bot}) \\
& = \spn\{ \id_N\}  \oplus \spn\{ \id_N\} ^{\perp_i},
\end{align*}
Since $\spn\{ \id_N\} ^{\perp_i}= \spn\{ \id_{[n_i]}\} \otimes \spn\{\id_{n_i}\}^{\bot},$ we claim that it is the orthogonal complement of the linear subspace generated by the identity matrix $\id_N = \id_{[n_i]} \otimes \id_{n_i}$ in the linear subspace $\id_{[n_i]} \otimes \R^{n_i\times n_i}.$ To prove the claim, observe that for $\id_{[n_i]} \otimes A_i \in \spn\{ \id_N\} ^{\perp_i}$ $(1 \le i \le d),$ by using \eqref{eq1}, it holds
$$
\langle \id_{[n_i]} \otimes A_i, \id_{N} \rangle_{\R^{N \times N}} = \langle \id_{[n_i]} \otimes A_i, \id_{[n_i]} \otimes \id_{n_i} \rangle_{\R^{N \times N}} = \tr(A_i) \prod_{\substack{j=1\\j \neq i}}^d n_j = 0,
$$ 
because $A_i \in \spn\{\id_{n_i}\}^{\bot}$ and hence $\tr(A_i)=0,$ for $1 \le i \le d.$
Thus, the claim follows and 
\begin{equation*}
\begin{matrix}
\begin{aligned}
   \spn \{\id_N\}^{\perp_i}& = \{\id_{[n_i]} \otimes A_i \in \spn\{\id_{[n_i]}\} \otimes \R^{n_i \times n_i}\, :\, \tr(A_i)=0\} \\ 
  &  = \left\{\id_{[n_i]} \otimes A_i \in \spn\{ \id_{[n_i]}\} \otimes \R^{n_i\times n_i}:
\langle \id_{[n_i]} \otimes A_i, \id_{N} \rangle_{\R^{N \times N}} = 0
\right\}.
\end{aligned}
\end{matrix}
\end{equation*}
\medskip

To prove \eqref{eq1thdelta}, we first consider $1 \le i < j \le d,$ and take $\id_{[n_k]} \otimes A_k \in \spn\{ \id_N\} ^{\perp_k}$  for $k=i,j.$ Then the inner product satisfies
\begin{align*}
\langle \id_{[n_i]} \otimes A_i,\id_{[n_j]} \otimes A_j \rangle_{\R^{N \times N}} & = \tr\left(
(\id_{[n_i]} \otimes A_i)^T(\id_{[n_j]} \otimes A_j)
\right) = \tr\left(
(\id_{[n_i]} \otimes A_i^T)(\id_{[n_j]} \otimes A_j)
\right)\\
& = \tr\left(\id_{n_1} \otimes \id_{n_{i-1}}\otimes A_i^T \otimes \id_{n_{i+1}} \otimes \cdots \otimes \id_{n_{j-1}} \otimes A_j \otimes \id_{n_{j+1}} \otimes \cdots \otimes \id_{n_d} \right)\\
 & = \prod_{\stackrel{\ell=1}{\ell \neq i,j}}^d \langle \id_{\ell},\id_{\ell}\rangle_{\R^{n_{\ell}\times n_{\ell}}} \tr(A_i) \tr(A_j) =0,
\end{align*}
because $ \tr(A_i) = \tr(A_j)=0.$ The same equality holds for $j < i.$ Thus, we conclude that $\spn\{ \id_N\} ^{\perp_i}$ is orthogonal to $\spn\{ \id_N\} ^{\perp_j}$
for all $i\neq j.$ So, the subspace 
$$
\Delta^{\prime} = \bigoplus_{i=1}^d \spn \{\id_N\}^{\perp_i},
$$ 
is well defined and it is a subspace of $\mathcal{L}(\mathbb{R}^{N \times N}).$ 
\medskip

%Por otro lado / We can write
%\begin{equation*}
%    \begin{matrix}
%    \begin{aligned}
%    \R^{N \times N} &= \R^{n_1 \times n_1} \otimes \R^{n_2 \times n_2} \otimes \dots \otimes \R^{n_d \times n_d} \\ &= \left( \R^{n_1\times n_1} \otimes \id_{[n_1]} \right) \cdot \left( \R^{n_2\times n_2} \otimes \id_{[n_2]} \right) \cdots  \left( \R^{n_d\times n_d} \otimes \id_{[n_d]} \right).
%    \end{aligned}
%    \end{matrix}
%\end{equation*}
To conclude the proof \eqref{eq1thdelta}, we will show that $\Delta' = \Delta.$ Since, for each $1 \le i \le d,$ $\spn \{\id_N\}^{\perp_i}$ is orthogonal to $ \spn \{\id_N\}$ we have 
$$
 \spn \{\id_N\} \oplus \Delta^{\prime} \subset \mathcal{L}(\mathbb{R}^{N \times N}).
$$
To obtain the equality, take $A \in \mathcal{L}(\mathbb{R}^{N \times N}).$ Then there exists $A_i \in \R^{n_i \times n_i}$
for $1 \le i \le d$ be such that
$$
A = \sum_{i=1}^d \mathrm{id}_{[n_i]} \otimes A_{i}.
$$
From Lemma~\ref{decomposition} we can write
$$
A_i = \frac{\tr(A_i)}{n_i}\, \id_{n_i} + \left( A_i  - \frac{\tr(A_i)}{n_i}\, \id_{n_i}\right)
$$
for each $1 \le i \le d.$ Then, 
\begin{align*}
A & = \sum_{i=1}^d \mathrm{id}_{[n_i]} \otimes \left( \frac{\tr(A_i)}{n_i}\, \id_{n_i} + \left( A_i  - \frac{\tr(A_i)}{n_i}\, \id_{n_i}\right)\right) \\ 
& = \sum_{i=1}^d \mathrm{id}_{[n_i]} \otimes \frac{\tr(A_i)}{n_i}\, \id_{n_i} +
\sum_{i=1}^d \mathrm{id}_{[n_i]} \otimes  \left( A_i  - \frac{\tr(A_i)}{n_i}\, \id_{n_i}\right) \\ 
& = \sum_{i=1}^d  \frac{\tr(A_i)}{n_i}\, \mathrm{id}_{[n_i]} \otimes \id_{n_i} +
\sum_{i=1}^d \mathrm{id}_{[n_i]} \otimes  \left( A_i  - \frac{\tr(A_i)}{n_i}\, \id_{n_i}\right) \\ 
& =  \left(\sum_{i=1}^d \frac{\tr(A_i)}{n_i}\right)\, \id_{N} + \sum_{i=1}^d \mathrm{id}_{[n_i]} \otimes  \left( A_i  - \frac{\tr(A_i)}{n_i}\, \id_{n_i}\right).
\end{align*}
Observe that $\left(\sum_{i=1}^d \frac{\tr(A_i)}{n_i}\right)\, \id_{N}  \in \spn\{\id_N\}$ and
$$
\sum_{i=1}^d \mathrm{id}_{[n_i]} \otimes  \left( A_i  - \frac{\tr(A_i)}{n_i}\, \id_{n_i}\right) \in \Delta^{\prime}.
$$
Thus, $\mathcal{L}(\mathbb{R}^{N \times N}) \subset \spn \{\id_N\} \oplus \Delta^{\prime}.$ In consequence $\Delta^{\prime} = \Delta,$ and this proves the theorem.
\end{proof}

A direct consequence of the above theorem is the next corollary.

\begin{corollary}\label{optimal_decom1}
Assume $\mathbb{R}^{N \times N},$ with $N=n_1 \cdots n_d \in \N.$ Then
$$
P_{\mathcal{L}(\R^{N \times N})} = P_{\spn\{\id_N\}} + \sum_{i=1}^d P_{\spn\{ \id_{N} \}^{\perp_i}},
$$
that is, for all $A \in \R^{N \times N}$ it holds
$$
P_{\mathcal{L}(\R^{N \times N})}(A) = \frac{\tr(A)}{N}\, \id_N + \sum_{i=1}^d \id_{n_1} \otimes \dots \id_{n_{i-1}} \otimes A_i \otimes \id_{n_{i+1}} \otimes \dots \otimes \id_{n_d},
$$
where $A_i \in \mathbb{R}^{n_i \times n_i}$ satisfies $\tr(A_i) =0$ for $1 \le i \le d.$ 
\end{corollary} 

\section{A Numerical Strategy to perform a Laplacian-like decomposition}\label{results}

Now, in this section we will study some numerical strategies in order to compute, for a
given matrix $A \in \R^{N \times N},$  with the help of Proposition~\ref{optimal_decom} and Theorem~\ref{thdelta}, its best Laplacian-like approximation. We start with
 the following Greedy Algorithm.

\begin{theorem}\label{thlaplacian}
    Let $A$ be a matrix in $ \R^{N \times N}$, with $N=n_1\cdots n_d,$ such that $\tr(A) = 0.$ Consider the following iterative
    procedure:
    \begin{enumerate}
	\item Take $X_k^{(0)} = 0$ for $1\le k \le d.$ 
	\item For each $\ell \ge 1$ compute for $1 \le i \le d$ the matrix $U_i^{(\ell)}$ as
	$$U_i^{(\ell)}=\arg \min_{U_i \in  \spn\{\id_{n_i}\}^{\perp}} 
		\left\| A - \sum_{k=1}^{i-1} \id_{[n_k]} \otimes X_k^{(\ell)} + \id_{[n_i]} \otimes (X_i^{(\ell-1)}+U_i) +  \sum_{k=i+1}^{d} \id_{[n_k]} \otimes X_k^{(\ell-1)} \right\|,$$
	and put $X_i^{(\ell)} = X_i^{(\ell-1)} + U_i^{(\ell)}.$
\end{enumerate}
Then
    $$
    \lim_{\ell \rightarrow \infty} \sum_{k=1}^d \id_{[n_k]} \otimes X_k^{(\ell)} = P_{\Delta}(A)
    $$
where $P_{\Delta}(A)$ is the orthogonal projection of $A$ on $\Delta=\bigoplus_{i=1}^d  \spn\{ \id_{ N} \}^{\perp_i}.$
\end{theorem}
%The above result states that if we write $A$ as $ A = L_A + R_A,$ then $L_A$ solves the problem
%	$$
%	\min_{A^* \in \Delta} \left\| A-A^* \right\|.
%	$$
\begin{proof}
Recall that $P_{\Delta}(A)$ solves the problem
	$$
	\min_{A^* \in \Delta} \left\| A-A^* \right\|.
	$$
To simplify notation put $P^{(\ell)}_{\Delta}(A) = \sum_{k=1}^d \id_{[n_k]} \otimes X_k^{(\ell)}$ for $\ell \ge 0.$
By construction we have that
	    $$
	    \| A-P_{\Delta}^{(1)}(A) \| \geq \| A-P_{\Delta}^{(2)}(A) \| \geq \dots \geq \| A-P_{\Delta}^{(\ell)}(A) \| \geq \cdots \geq 0,
	    $$
	    holds. Since the sequence $\{ P_{\Delta}^{(\ell)}(A) \}_{\ell \in \N}$ is bounded, there is a convergent subsequence also denoted by $\{ P_{\Delta}^{(\ell)}(A) \}_{\ell \in \N}$, so that
	    $$
	    L_A = \lim\limits_{\ell \to \infty} P_{\Delta}^{(\ell)}(A) \in \Delta.
	    $$
		If $L_A = P_{\Delta}(A)$, the theorem holds. Otherwise, assume that $L_A \neq P_{\Delta}(A),$ 
	    then it is clear that 
	    $$
	    \| A- P_{\Delta}(A)\| \leq \| A - L_A \|.
	    $$
	    Suppose that $ \| A-P_{\Delta}(A) \| < \| A - L_A \|$ and let $\lambda \in (0,1).$  Now, 
	    consider the linear combination $\lambda L_A+(1-\lambda)P_{\Delta}(A).$ Since $L_A, P_{\Delta}(A) \in \Delta$, they can be written as
	    $$
	    L_A=\sum_{i=1}^d \id_{[n_i]} \otimes A_i, \quad \text{and} \quad P_{\Delta}(A)=\sum_{i=1}^d \id_{[n_i]} \otimes A_i^*,
	    $$
	    so $\lambda L_A+(1-\lambda)P_{\Delta}(A)=\sum_{i=1}^d \id_{[n_i]} \otimes \left( \lambda A_i + (1-\lambda)A_i^* \right) \in \Delta$. Hence,
	    $$
	    \| A-P_{\Delta}(A) \| < \| A-  \left( \lambda L_A+(1-\lambda)P_{\Delta}(A) \right) \| < \| A - L_A \|.
	    $$
	    \begin{figure}[h]
	        \centering
	       \includegraphics[width=0.9\linewidth]{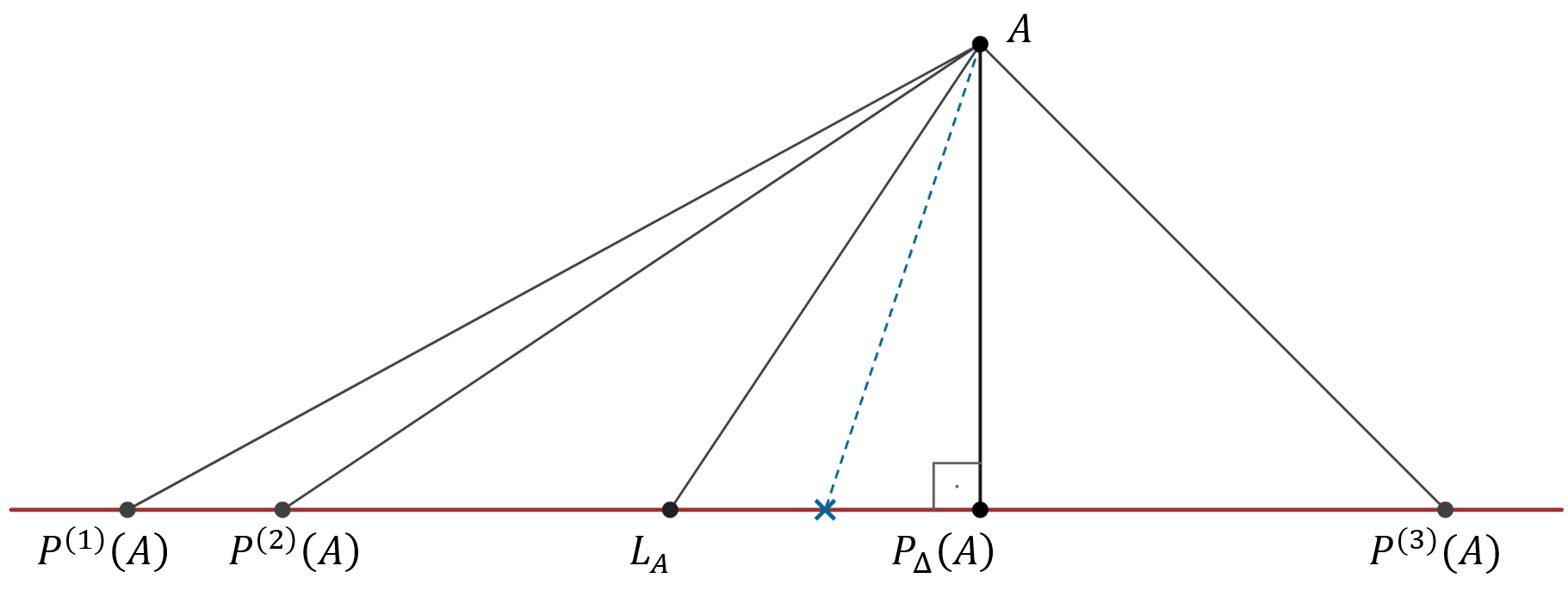}
	        \caption{Situation described in reasoning by R.A.A.}
	        %\label{fig:my_label}
	    \end{figure}
	    \noindent
	    That is, we have found $d$ matrices $Z_i=\lambda A_i + (1-\lambda)A_i^*$, $i=1,\dots,d$, such that
	    $$
	    \left\|A - L_A \right\|= \left\| A-\sum_{i=1}^d \id_{[n_i]} \otimes A_i\right\| > \left\| A-\sum_{i=1}^d \id_{[n_i]} \otimes Z_i \right\|,
	    $$
	    which is a contradiction with the definition of $L_A$.
\end{proof}

%%%%%%%%%%%%%%%%%%%%%%%%%%%%%%%%%%%%%%%
%               ALGORITHM
%%%%%%%%%%%%%%%%%%%%%%%%%%%%%%%%%%%%%%%

The previous result allows us to describe the procedure to obtain the Laplacian approximation of a square matrix, in the form of an algorithm. We can visualize the complete algorithm in the form of pseudocode in Algorithm \ref{code}.
\begin{algorithm}[h]
	\caption{Laplacian decomposition Algorithm}\label{code}
	\begin{algorithmic}[1] %el numerito es para que aparezca numerado el esquema
		\Procedure{Lap}{$A^*,\texttt{iter\_max},\texttt{tol}$} \vspace{0.1mm} 
		\State $A = A^* - (\tr(A)/N) \id_N$, $\texttt{iter}=1$, $\texttt{Lap} = 0$ \vspace{0.1mm}
		%\State $\texttt{Lap} = 0$
		\While{$\texttt{iter} < \texttt{iter\_max}$} \vspace{0.1mm}
		\State $A \leftarrow A - \texttt{Lap} $ \vspace{0.1mm}
		\For{$k=1,2,\ldots,d$} \vspace{0.1mm}
		\State $P_k(A)=\id_{n_1}\otimes \dots \otimes \id_{n_{k-1}} \otimes X_k \otimes \id_{n_{k+1}} \otimes \dots \otimes \id_{n_d}$
		\State $X_k \leftarrow \min_{X_k} \|A - \sum_{i=1}^k P_i(A) \|$ \vspace{0.1mm}
		\State $\texttt{Lap} = \texttt{Lap} + P_k(A)$ \vspace{0.1mm}
		\EndFor \vspace{0.1mm}
		\If {$\| A - \texttt{Lap}\| < \texttt{tol}$} \textbf{goto} 14 \vspace{0.1mm} %the residue
		\EndIf \vspace{0.1mm}
		\State $\texttt{iter}=\texttt{iter}+1$ \vspace{0.1mm}
		\EndWhile \vspace{0.1mm}
		\State \textbf{return} $\texttt{Lap}$ \vspace{0.1mm}
		\EndProcedure
	\end{algorithmic}
\end{algorithm}

%%%%%%%%%%%%%%%%%%%%%%%%%%%%%%%%%%%%%%%
%               EJEMPLOS
%%%%%%%%%%%%%%%%%%%%%%%%%%%%%%%%%%%%%%%
\subsection{Numerical Examples}
\subsubsection{Example 1: The adjacency matrix of a simple graph}
First, let us show an example in which the projection $P_{\Delta}(A)$ coincides with $A$ and how the tensor representations is provided by the aforementioned proposed algorithm.
Let us consider the simple graph $G(V,E)$, with $V=\{1,2,\ldots,6\}$ the set of nodes and $E=\{(1,4),(2,3),(2,5),(3,6),(5,6)\}$ the set of edges. Then, the adjacency matrix of $G$ is %\small
\begin{equation*} 
    A = 
	\begin{pmatrix}
	0 & 1 & 0 & 1 & 0 & 0 \\
	1 & 0 & 1 & 0 & 1 & 0 \\
	0 & 1 & 0 & 0 & 0 & 1 \\
	1 & 0 & 0 & 0 & 1 & 0 \\
	0 & 1 & 0 & 1 & 0 & 1 \\
	0 & 0 & 1 & 0 & 1 & 0 
	\end{pmatrix}.
\end{equation*}
%\normalsize
We want to find a Laplacian decomposition of the matrix $A \in \R^{6 \times 6}$. Since $\tr(A)=0$, we can do this by following the iterative scheme given by Theorem \ref{thlaplacian}. So, we look for $X_1 \in \R^{2 \times 2}$, $X_2 \in \R^{3 \times 3}$ matrices such that
$$
P_{\Delta}(A)=X_1 \otimes \id_{n_2}+\id_{n_1} \otimes X_2,
$$
where $n_1=2, n_2=3$. We proceed according to the algorithm: 
\begin{enumerate}
	\item[1.] Compute $X_1:$
    \begin{equation*}
    \min_{X_1} \| A- X_1 \otimes \id_{n_2} \| \, \Rightarrow \, X_1 = 
        \begin{pmatrix}
        0 & 1 \\
        1 & 0
        \end{pmatrix}.
    \end{equation*} 
    
    \item[2.] Compute $X_2:$
    \begin{equation*}
    \begin{matrix}
    \begin{aligned}
    %\hspace{-5mm}
    \min_{X_2} \| A - P_1(A) - \id_{n_1} \otimes X_2\| \, \Rightarrow \, X_2 =     \begin{pmatrix}
        0& 1 & 0 \\
        1 & 0 & 1\\
        0 & 1 & 0
        \end{pmatrix}.
    \end{aligned}
    \end{matrix}
    \end{equation*} 
\end{enumerate}
Since the residual values is $||A-P_{\Delta}(A)||=0,$ the matrix $A \in \Delta$ and we can write it as
$$
A = \begin{pmatrix}
0 & 1 \\
1 & 0
\end{pmatrix} \otimes \id_{n_2} + \id_{n_1} \otimes 
\begin{pmatrix}
0& 1 & 0 \\
1 & 0 & 1\\
0 & 1 & 0
\end{pmatrix} = P_{\Delta}(A).
$$

\subsubsection{Example 2: A bigger sparse matrix}
Now, let us consider the following sparse matrix in $\GL(\R^{30})$,
	\begin{equation*}
	A=
	\left(
	\begin{array}{rrr|rrr}
	    T_1 & 2I_5 & -I_5 &  & & \\
		2I_5 & T_2 & 2I_5 & & I_{15} & \\ 
	    I_5 & 2I_5 & T_1 & & & \\
		\hline
	    & & & T_1 & 2I_5 & -I_5 \\
		& -I_{15} & & 2I_5 & T_2 & 2I_5 \\ 
	    & & &  I_5 & 2I_5 & T_1 \\
	\end{array}
	\right),
	\end{equation*}
	where
	\begin{equation*}
	T_1 =
    \begin{pmatrix*}[r]
     3 & \phantom{-}2 & \phantom{-}1 & \phantom{-}0 & -2 \\
    2 & 3 & 2 & 1 & 0 \\
    1 & 2 & 3 & 2 & 1 \\
    0 & 1 & 2 & 3 & 2 \\
    -2 & 0 & 1 & 2 & 3 \\
    \end{pmatrix*}, \; T_2 =
    \begin{pmatrix*}[r]
     -1 & 2 & 1 & 0 & -2 \\
     2 & -1 & 2 & 1 & 0 \\
     1 & 2 & -1 & 2 & 1 \\
     0 & 1 & 2 & -1 & 2 \\
    -2 & 0 & 1 & 2 & -1 \\
    \end{pmatrix*}
\end{equation*}
\medskip\noindent
and $I_k$ es the identity matrix of size $k \times k$. We can visualize the matrix graphically with the \textit{Mathematica} command $\texttt{MatrixPlot[A]}$,
In this case $\tr(A)=50 \neq 0,$ so instead of looking for the Laplacian approximation of $A$, we will look for it of $$\hat{A}=\left(A-\dfrac{\tr(A)}{30}\id_{30}\right),$$ which has null trace; see Figure \ref{fig:sparse_mathematica}.

\begin{figure}[h]
    \centering
    \hspace{-5mm}
	\includegraphics[width=0.6\linewidth]{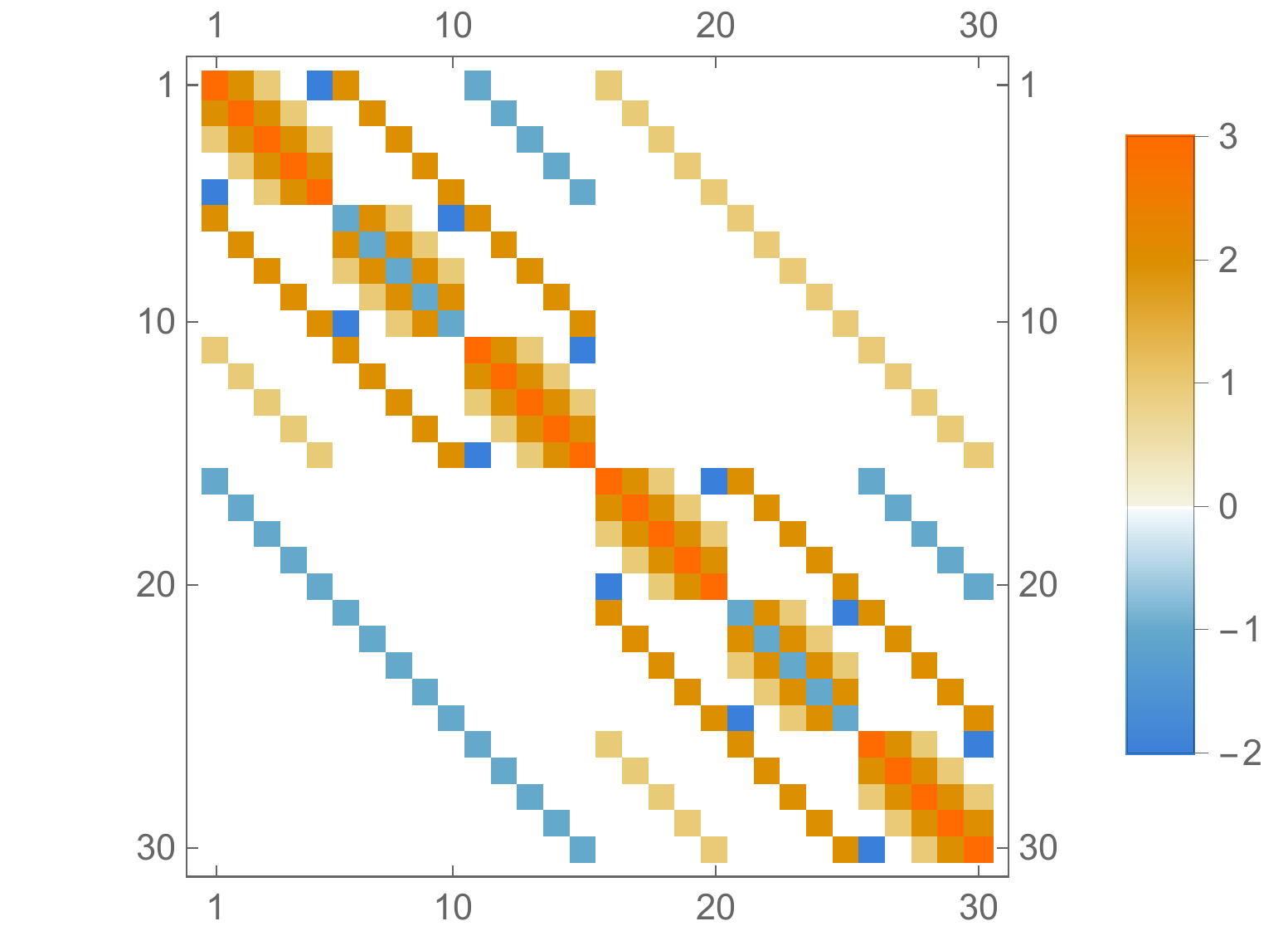}
	\caption{Representation of sparse matrix $A$ using \textit{Mathematica} \cite{Mathematica}.}\label{fig:sparse_mathematica}
\end{figure}\medskip

Again, we proceed according to the algorithm:
\begin{enumerate}
	\item[1.] Compute $X_1:$
    \begin{equation*}
    \min_{X_1} \| \hat{A}- X_1 \otimes \id_{n_2}\otimes \id_{n_3} \| \, \Rightarrow \, X_1 = 
        \begin{pmatrix}
        \phantom{-}0 & 1 \\
        -1 & 0
        \end{pmatrix}.
    \end{equation*} 
    
    \item[2.] Compute $X_2:$
    \begin{equation*}
    \min_{X_2} \| \hat{A} - X_1 \otimes \id_{n_2}\otimes \id_{n_3} - \id_{n_1} \otimes X_2 \otimes \id_{n_3}\| \, \Rightarrow \, X_2 =     
        \begin{pmatrix}
        \frac{4}{3}& \phantom{-}2 & -1 \\
        2 & -\frac{8}{3} & \phantom{-}2\\
        1 & \phantom{-}2 & \phantom{-}\frac{4}{3}
        \end{pmatrix}.
    \end{equation*} 
    
    \item[3.] Compute $X_3:$
    \begin{equation*}
    \begin{matrix}
    \begin{aligned}
    %\hspace{-5mm}
     \min_{X_3} \| \hat{A} - X_1 \otimes \id_{n_2}\otimes \id_{n_3} - \id_{n_1} \otimes X_2 \otimes \id_{n_3}- \id_{n_1}  \otimes \id_{n_2} \otimes X_3\| \\ & \Rightarrow \, X_3 =     \begin{pmatrix}
        \phantom{-}0 &   \phantom{-}2 &   \phantom{-}1 &   \phantom{-}0 & -2 \\
        \phantom{-}2 &   \phantom{-}0 &   \phantom{-}2 &   \phantom{-}1 & \phantom{-}0 \\
        \phantom{-}1 &   \phantom{-}2 &   \phantom{-}0 &   \phantom{-}0 & \phantom{-}2 \\
        \phantom{-}0 &   \phantom{-}1 &   \phantom{-}2 &   \phantom{-}0 & \phantom{-}2 \\
        -2 & \phantom{-}0 & \phantom{-}1 & \phantom{-}2 & \phantom{-}0
        \end{pmatrix}.
    \end{aligned}
    \end{matrix}
    \end{equation*} 
\end{enumerate}\medskip
The residue of the approximation of $\hat{A}$ is
$
\| \hat{A} - P_{\Delta}(\hat{A}) \| = 0,
$
so, following Corollary~\ref{optimal_decom1}, we can write the original matrix $A$ as 
$$
A = \dfrac{\tr(A)}{30}\id_{30} + X_1 \otimes \id_{n_2}\otimes \id_{n_3} + \id_{n_1} \otimes X_2 \otimes \id_{n_3}+ \id_{n_1}  \otimes \id_{n_2} \otimes X_3.
$$
\medskip
Note that the first term is
$$
\dfrac{\tr(A)}{30}\id_{30} = \begin{pmatrix}
\frac{5}{3} & 0 \\
0 & \frac{5}{3}
\end{pmatrix} \otimes \id_{n_2} \otimes \id_{n_3},
$$
\medskip\noindent
and hence $A$ can be written as:
\begin{equation*}
    \begin{matrix}
    \begin{aligned}
    A = \begin{pmatrix}
        \frac{5}{3} & 1 \\
        -1 & \frac{5}{3}
        \end{pmatrix} \otimes \id_{n_2 n_3} + \id_{n_1} \otimes \begin{pmatrix}
        \frac{4}{3}& \phantom{-}2 & -1 \\
        2 & -\frac{8}{3} & \phantom{-}2\\
        1 & \phantom{-}2 & \phantom{-}\frac{4}{3}
        \end{pmatrix} \otimes \id_{n_3} + \id_{n_1 n_2} \otimes \begin{pmatrix}
        \phantom{-}0 &   \phantom{-}2 &   \phantom{-}1 &   \phantom{-}0 & -2 \\
        \phantom{-}2 &   \phantom{-}0 &   \phantom{-}2 &   \phantom{-}1 & \phantom{-}0 \\
        \phantom{-}1 &   \phantom{-}2 &   \phantom{-}0 &   \phantom{-}0 & \phantom{-}2 \\
        \phantom{-}0 &   \phantom{-}1 &   \phantom{-}2 &   \phantom{-}0 & \phantom{-}2 \\
        -2 & \phantom{-}0 & \phantom{-}1 & \phantom{-}2 & \phantom{-}0
        \end{pmatrix}. \normalsize
    \end{aligned}
    \end{matrix}
\end{equation*}
 
\section{Conclusions}\label{conclusions}
We have presented a result to approximate a generic square matrix by its Laplacian form, and thus decompose it as the sum of two linearly independent matrices. This decomposition is motivated by the fact that tensor algorithms are more efficient when working with Laplacian matrices. We have also described the procedure to perform this approximation in the form of an algorithm and illustrated how it works on some basic examples.\medskip  %...which allows us to approximate the solution of the system using different methods, like the Greedy Rank-One algorithm.  \vspace{2mm}

%In short, 
With the proposed algorithm we may provide an alternative way to solve linear system $ A\xx = \bb $. Due to its structure, this matrix decomposition can be interesting when studying sparse matrices, matrices that come from the discretization of a PDE, or adjacency matrices of simple graphs, among others. We will explore the computational gains of this approach in  different contexts in forthcoming works.

%%%%%%%%%%%%%%%%%%%%%%%%%%%%%%%%%%%%%%%%%%%%%%%%%%%%%%
%%%%%%%%%%%%%%%%%%%%%%%%%%%%%%%%%%%%%%%%%%%%%%%%%%%%%5

%%%%%%%%%%%%%%%%%%%%%%%%%%%%%%%%%%%%%%%%%%%%%%%%%%%%%%
%%%%%%%%%%%%%%%%%%%%%%%%%%%%%%%%%%%%%%%%%%%%%%%%%%%%%5

%\backmatter

\section*{Acknowledgments} %Funding information
This work was supported by the Generalitat Valenciana and the European Social Found under Grant [number ACIF/2020/269)]; Ministerio de Ciencia, Innovación y Universidades under Grant [number RTI2018-093521-B-C32];  Universidad CEU Cardenal Herrera under Grant [number  INDI20/13].

%\subsection*{Author contributions}
%This is an author contribution text. This is an author contribution text. This is an author contribution text. This is an author contribution text. This is an author contribution text. 

\subsection*{Conflict of interest}
The authors declare no potential conflict of interests.

%\section*{Supporting information}

%The following supporting information is available as part of the online article:

%\noindent
%\textbf{Figure S1.}
%{500{\uns}hPa geopotential anomalies for GC2C calculated against the ERA Interim reanalysis. The period is 1989--2008.}

%\noindent
%\textbf{Figure S2.}
%{The SST anomalies for GC2C calculated against the observations (OIsst).}

%%%%%%%%%%%%%%%%%%%%%%%%%%%%%%%%%%%%%%%
%              REFERENCIAS
%%%%%%%%%%%%%%%%%%%%%%%%%%%%%%%%%%%%%%%

\nocite{*}% Show all bib entries - both cited and uncited; comment this line to view only cited bib entries;
\bibliography{wileyNJD-VANCOUVER}%

%\section*{Author Biography}

%\begin{biography}{\includegraphics[width=66pt,height=86pt,draft]{empty}}{\textbf{Author Name.} this is sample author biography text this is sample author biography text.}
%\end{biography}

\end{document}